\documentclass[3p,number,sort&compress]{elsarticle}

\journal{Discrete Applied Mathematics}

\makeatletter
\def\ps@pprintTitle{\def\@oddfoot{\textit{\footnotesize Submitted to Discrete Applied Mathematics}\hfill\textit{\footnotesize 4 February 2016}}}
\makeatother

\usepackage{amsthm}
\usepackage{lineno,hyperref}
\usepackage{amsmath,amssymb,mathabx,verbatim,graphicx,color,bbm}
\usepackage{fixme,revsymb4-1}
\usepackage{verbatim,color,xcolor}

\newtheorem{theo}{Theorem} 
\newtheorem{lem}[theo]{Lemma}

\newtheorem{defn}[theo]{Definition}

\def\squareforqed{\hbox{\rlap{$\sqcap$}$\sqcup$}}
\def\qed{\ifmmode\squareforqed\else{\unskip\nobreak\hfil
\penalty50\hskip1em\null\nobreak\hfil\squareforqed
\parfillskip=0pt\finalhyphendemerits=0\endgraf}\fi}
\def\endenv{\ifmmode\;\else{\unskip\nobreak\hfil
\penalty50\hskip1em\null\nobreak\hfil\;
\parfillskip=0pt\finalhyphendemerits=0\endgraf}\fi}
\newenvironment{remark}{\noindent \textbf{{Remark~}}}{}

\mathchardef\ordinarycolon\mathcode`\:
\mathcode`\:=\string"8000
\def\vcentcolon{\mathrel{\mathop\ordinarycolon}}
\begingroup \catcode`\:=\active
  \lowercase{\endgroup
  \let :\vcentcolon
  }

\newcommand{\nc}{\newcommand}
\nc{\rnc}{\renewcommand}
\nc{\beq}{\begin{equation}}
\nc{\eeq}{\end{equation}}
\nc{\beqa}{\begin{eqnarray}}
\nc{\eeqa}{\end{eqnarray}}
\nc{\lbar}[1]{\overline{#1}}
\nc{\bra}[1]{\langle#1|}
\nc{\ket}[1]{|#1\rangle}
\nc{\ketbra}[2]{|#1\rangle\!\langle#2|}
\nc{\braket}[2]{\langle#1|#2\rangle}
\nc{\proj}[1]{| #1\rangle\!\langle #1 |}
\nc{\avg}[1]{\langle#1\rangle}
\nc{\Rank}{\operatorname{Rank}}
\nc{\smfrac}[2]{\mbox{$\frac{#1}{#2}$}}
\nc{\tr}{\operatorname{Tr}}
\nc{\ox}{\otimes}
\nc{\dg}{\dagger}
\nc{\dn}{\downarrow}
\nc{\cA}{{\cal A}}
\nc{\cB}{{\cal B}}
\nc{\cC}{{\cal C}}
\nc{\cD}{{\cal D}}
\nc{\cE}{{\cal E}}
\nc{\cF}{{\cal F}}
\nc{\cG}{{\cal G}}
\nc{\cH}{{\cal H}}
\nc{\cI}{{\cal I}}
\nc{\cJ}{{\cal J}}
\nc{\cK}{{\cal K}}
\nc{\cL}{{\cal L}}
\nc{\cM}{{\cal M}}
\nc{\cN}{{\cal N}}
\nc{\cO}{{\cal O}}
\nc{\cP}{{\cal P}}
\nc{\cQ}{{\cal Q}}
\nc{\cR}{{\cal R}}
\nc{\cS}{{\cal S}}
\nc{\cT}{{\cal T}}
\nc{\cX}{{\cal X}}
\nc{\cY}{{\cal Y}}
\nc{\cZ}{{\cal Z}}
\nc{\csupp}{{\operatorname{csupp}}}
\nc{\qsupp}{{\operatorname{qsupp}}}
\nc{\var}{{\operatorname{var}}}
\nc{\rar}{\rightarrow}
\nc{\lrar}{\longrightarrow}
\nc{\polylog}{{\operatorname{polylog}}}
\nc{\1}{{\openone}}
\nc{\wt}{{\operatorname{wt}}}
\nc{\av}[1]{{\left\langle {#1} \right\rangle}}

\nc{\RR}{{{\mathbb R}}}
\nc{\CC}{{{\mathbb C}}}
\nc{\FF}{{{\mathbb F}}}
\nc{\NN}{{{\mathbb N}}}
\nc{\ZZ}{{{\mathbb Z}}}
\nc{\PP}{{{\mathbb P}}}
\nc{\QQ}{{{\mathbb Q}}}
\nc{\UU}{{{\mathbb U}}}
\nc{\EE}{{{\mathbb E}}}
\nc{\id}{{\operatorname{id}}}

\nc{\CHSH}{{\operatorname{CHSH}}}

\nc{\be}{\begin{equation}}
\nc{\ee}{{\end{equation}}}
\nc{\bea}{\begin{eqnarray}}
\nc{\eea}{\end{eqnarray}}
\nc{\<}{\langle}
\rnc{\>}{\rangle}
\nc{\Hom}[2]{\mbox{Hom}(\CC^{#1},\CC^{#2})}
\nc{\rU}{\mbox{U}}

\nc{\ob}[1]{#1}

\begin{document}

\begin{frontmatter}

\title{A new property of the Lov\'asz number \protect\\ 
       and duality relations between graph parameters\tnoteref{mytitlenote}}

\tnotetext[mytitlenote]{Special issue in memory of Levon Khachatrian, 1954--2004}

\author[ICREAaddress,ICFOaddress]{Antonio Ac\'in}

\author[UTSaddress,Tsinghuaaddress]{Runyao Duan}

\author[UCLaddress]{David E. Roberson}

\author[ICFOaddress,UoBaddress]{Ana Bel\'en Sainz}

\author[ICREAaddress,UABaddress]{Andreas Winter}
\ead{andreas.winter@uab.cat}

\address[ICREAaddress]{ICREA -- Instituci\'{o} Catalana de Recerca i Estudis Avan\c{c}ats, 
                       Pg.~Llu\'{\i}s Companys, 23, 08010 Barcelona, Spain}

\address[ICFOaddress]{ICFO -- Institut de Ciencies Fotoniques, 
                      The Barcelona Institute of Science and Technology, \protect\\
                      08860 Castelldefels (Barcelona), Spain}

\address[UTSaddress]{Centre for Quantum Computation and Intelligent Systems (QCIS), 
                     Faculty of Engineering and Information Technology, 
                     University of Technology Sydney, Sydney, NSW 2007, Australia}

\address[Tsinghuaaddress]{State Key Laboratory of Intelligent Technology and Systems, 
                 Tsinghua National Laboratory for Information Science and Technology, 
                 Department of Computer Science and Technology, Tsinghua University, Beijing 100084, China}



\address[UCLaddress]{Department of Computer Science, University College London, Gower Street, 
                     London WC1E 6BT, United Kingdom}

\address[UoBaddress]{School of Physics, H H Wills Physics Laboratory, 
                     University of Bristol, Tyndall Avenue, Bristol BS8 1TL, U.K.}

\address[UABaddress]{F\'{\i}sica Te\`{o}rica: Informaci\'{o} i Fen\`{o}mens Qu\`{a}ntics, 
                     Universitat Aut\`{o}noma de Barcelona, 08193 Bellaterra (Barcelona), Spain}

\begin{abstract}
  We show that for any graph $G$,
  by considering ``activation'' through the strong product
  with another graph $H$, the relation $\alpha(G) \leq \vartheta(G)$
  between the independence number and the Lov\'{a}sz number of $G$
  can be made arbitrarily tight: Precisely, the inequality
  \[
    \alpha(G \boxtimes H) \leq \vartheta(G \boxtimes H) = \vartheta(G)\,\vartheta(H)
  \]
  becomes asymptotically an equality for a suitable sequence of 
  ancillary graphs $H$.
  
  This motivates us to look for other products of graph parameters of
  $G$ and $H$ on the right hand side of the above relation. For instance,
  a result of Rosenfeld and Hales states that
  \[
    \alpha(G \boxtimes H) \leq \alpha^*(G)\,\alpha(H),
  \]
  with the fractional packing number $\alpha^*(G)$, and for every $G$ there
  exists $H$ that makes the above an equality; conversely,
  for every graph $H$ there is a $G$ that attains equality.
  
  These findings constitute some sort of duality of graph parameters, mediated
  through the independence number, under which $\alpha$ and $\alpha^*$
  are dual to each other, and the Lov\'{a}sz number $\vartheta$ is
  self-dual. We also show duality of Schrijver's and Szegedy's variants 
  $\vartheta^-$ and $\vartheta^+$ of the Lov\'{a}sz number, 
  and explore analogous notions for the chromatic number under strong and
  disjunctive graph products. 
\end{abstract}

	
\end{frontmatter}

\section{Independence number of a graph and its relaxations}
\label{sec:intro}
In the present paper we consider graphs $G=(V,E)$, which throughout will be
undirected and without loops \cite{Berge1973}. We shall be using
the Lov\'{a}sz convention~\cite{Lovasz1979}, writing $v\sim w$ to denote
$vw\in E$ or $v=w$.
We shall be concerned with various graph parameters, starting
from the \emph{independence number} (aka \emph{stability number}
or \emph{packing number})
\begin{equation}
  \label{eq:alpha}
  \alpha(G) = \max |I| \text{ s.t. } I\subset V \text{ is an independent set},
\end{equation}
where $I$ is called an \emph{independent} (or \emph{stable}) \emph{set} 
if the induced graph $G|_I$ 
is a graph with no edges, i.e.~the complement of the complete graph on the
vertices $I$. Computing $\alpha$ is well-known to be NP-complete \cite{Karp1972}.

In the present paper we are interested in how the independence number
behaves under product composition of graphs $G=(V,E)$ and $H=(V',E')$. We will consider
the \emph{strong product} $G\boxtimes H$ and the \emph{disjunctive product} $G\ast H$. 
These two products have as vertex set the Cartesian product $V\times V'$, 
while the corresponding edge sets are defined as follows:
\begin{align*}
  (vv',ww') \in E(G\boxtimes H) &\text{ iff } v=w\ \&\ v'w'\in E' \text{ or }
                                               vw \in E\ \&\ v'=w' \text{ or }
                                               vw \in E\ \&\ v'w'\in E', \\
  (vv',ww') \in E(G\ast H)      &\text{ iff } vw \in E \text{ or } v'w'\in E'.
\end{align*}
The two graph products are related by a de Morgan identity:
$\overline{G\boxtimes H} = \overline{G}\ast\overline{H}$, which is
why they are sometimes called ``and'' ($\boxtimes$) and ``or'' ($\ast$) product.
They exhibit very different behaviour for the independence number:
\[
  \alpha(G \ast H) = \alpha(G)\,\alpha(H),
  \quad\text{ but }\quad
  \alpha(G \boxtimes H) \geq \alpha(G)\,\alpha(H),
\]
and the inequality is in general strict. 
E.g.~for the five-cycle (``pentagon'') $C_5$, we have $\alpha(C_5)=2$
but $\alpha(C_5\boxtimes C_5)=5$.

The independence number and the strong graph product were studied as early as 1956, 
in Shannon's seminal paper on zero-error communication \cite{Shannon1956}, 
in particular the asymptotic behaviour of $\alpha(G^{\boxtimes n}) \sim \Theta(G)^n$, 
where $G^{\boxtimes n} = G \boxtimes G \boxtimes \cdots \boxtimes G$, giving rise to the 
\emph{zero-error (Shannon) capacity} 
\[
  \Theta(G) = \sup_n \bigl( \alpha(G^{\boxtimes n}) \bigr)^{1/n}
\] 
of $G$. 
The strong graph product arises naturally in communication via noisy channels; 
indeed, if $G$ is the confusability graph of a channel, the confusability graph 
of $n$ independent uses of the channel is $G^{\boxtimes n}$.

In his paper, Shannon already introduced a useful
upper bound on $\alpha$ and $\Theta$, which was to become known as the
\emph{fractional packing number} and denoted $\alpha^*$ \cite{Shannon1956}. 
This bound has also been called \emph{Rosenfeld number} in the literature, 
perhaps because its appearance in Shannon's work was not fully appreciated. 
It is defined as
\begin{equation}
  \label{eq:alpha-star}
  \alpha^*(G) = \max \sum_v t_v \text{ s.t. } t_v\geq 0\ \forall v,\ 
                          \sum_{v\in C} t_v \leq 1\ \forall \text{ cliques }C\subset V.
\end{equation}
Here, by a \emph{clique} we mean a complete induced subgraph,
i.e.~$G|_C \simeq K_{m}$, $m=|C|$.
Eq.~(\ref{eq:alpha-star}) is a linear programme (LP), 
and hence efficiently computable once
the cliques are known. To be precise, Shannon had defined it more
generally for hypergraphs (cf.~\cite{Berge1973,ScheinermanUllman1997}), 
which is more natural for an actual communication channel with inputs and 
outputs; the definition above, which is the one whose study Rosenfeld 
initiated \cite{Rosenfeld1967}, is obtained for the hypergraph of all 
cliques of $G$.

In fact, for the clique hypergraph of $G$,
Shannon identified $\alpha^*(G)$ as the zero-error capacity assisted 
by instantaneous feedback of a channel with confusability graph $G$. 
In \cite{CLMW2011}, it was shown that $\alpha^*(G)$ is also the 
zero-error capacity assisted by so-called ``no-signalling'' correlations.
Both result extend to general channels and their hypergraphs,
see~\cite{Shannon1956,CLMW2011} for details.
Shannon furthermore conjectured that $\log \alpha^*(G)$ equals the minimum of the 
usual Shannon capacity over all noisy channels with confusability graph $G$,
which was proved later by Ahlswede \cite{Ahlswede1973};
see also \cite{DuanSeveriniWinter2014} for an alternative proof.
(Note that here the logarithm appears because in information theory
the capacity is measured in bits per channel use, while in zero-error
theory and combinatorics, it is defined via an $n$-th root.) 
All of these imply operational, information theoretic proofs of 
$\alpha(G) \leq \Theta(G) \leq \alpha^*(G)$. However, it can be seen also
in elementary fashion, noticing that restricting the variables in
eq.~(\ref{eq:alpha-star}) to values $\{0,1\}$ yields precisely
the independence number, so $\alpha(G) \leq \alpha^*(G)$. To get the
upper bound on $\Theta(G)$ as well, we use
\[
  \alpha^*(G \boxtimes H) = \alpha^*(G)\,\alpha^*(H),
\]
which follows from the primal and dual LP characterizations of
the fractional packing number (see \ref{app:optimization}).
In particular,
\[
  \alpha(G^{\boxtimes n}) \leq \alpha^*(G^{\boxtimes n}) = \bigl( \alpha^*(G) \bigr)^n,
\]
and the claim follows. For instance, $\alpha^*(C_5) = \frac52$
is an upper bound on $\Theta(C_5)$, but it is not tight.

It took more than twenty years to improve this bound significantly, with the 
discovery of Lov\'{a}sz that a semidefinite programme (SDP)
can emulate many of the nice properties of the fractional packing number:
\begin{equation}
  \label{eq:lovasz}
  \vartheta(G) = \max \tr BJ \text{ s.t. } B\geq 0,\ \tr B=1,\ B_{vw}=0\ \forall vw\in E
\end{equation}
(where $J$ is the all-ones matrix) is also an upper bound on $\alpha(G)$ and is multiplicative:
\[
  \vartheta(G \boxtimes H) = \vartheta(G \ast H) = \vartheta(G)\,\vartheta(H),
\]
hence $\Theta(G) \leq \vartheta(G)$ \cite{Lovasz1979}. Returning to
the pentagon, $\vartheta(C_5) = \sqrt{5} = \Theta(C_5)$.
For a selection of different characerizations of the Lov\'{a}sz number
see \ref{app:optimization}.

\medskip
The rest of the paper is structured as follows. 
In Section~\ref{sec:activation} we show that $\alpha$ and $\vartheta$
can be made asymptotically equal by taking the strong product with suitable
auxiliary graphs.
Then, in Section~\ref{sec:rosenfeld} we recall (and prove) a result
similar in spirit, due to Rosenfeld \cite{Rosenfeld1967} and Hales \cite{Hales1973}, 
which establishes a certain duality between $\alpha$ and $\alpha^*$.
In Section~\ref{sec:theta-plus-minus}, we go on to show 
a similar duality between Schrijver's and Szegedy's variants 
$\vartheta^{\pm}$ of the Lov\'{a}sz number.
Motivated by the \emph{Sandwich Theorem}, Section~\ref{sec:chromatic} 
is devoted to an investigation of analogous questions with the chromatic
number instead of the independence number. Throughout the text,
various remarks offer reflections on our findings and highlight
open problems.
Finally, in Section~\ref{sec:discussion} we conclude, discussing 
what we have learned and speculating on future directions.

\section{Finite and asymptotic activation attaining the Lov\'{a}sz number}
\label{sec:activation}
In general, $\alpha(G)$ is strictly smaller than $\vartheta(G)$
or indeed the integer part of the latter, and this persist even in the
many-copy asymptotics: there are graphs with $\Theta(G) < \vartheta(G)$
\cite{Haemers1979,Peeters1996}.

On the other hand, what we will show in this section is that going beyond 
graph products of the form $G^{\boxtimes n} = G \boxtimes G^{\boxtimes (n-1)}$, 
and considering general products $G\boxtimes H$, closes the gap between 
$\alpha$ and $\vartheta$.
Indeed, Lov\'asz \cite{Lovasz1979} already proved that for vertex-transitive 
$G = (V,E)$, i.e.~when the automorphism group of $G$ maps any vertex to 
any other one,
\[
  \vartheta(G \boxtimes \overline{G})
     = \vartheta(G)\,\vartheta(\overline{G}) 
     = |V|
     = \alpha(G \boxtimes \overline{G}).
\]
This begs the natural question whether for every graph $G$, there exists 
another graph $H$ such that 
\begin{equation}
  \label{eq:activation}
  \alpha(G \boxtimes H) = \vartheta(G \boxtimes H) = \vartheta(G)\,\vartheta(H)\text{?}
\end{equation}

It turns out that by allowing \emph{weighted} graphs $(H,p)$, the answer is yes,
even with $H=\overline{G}$:
\begin{lem}
  \label{lem:alpha-theta}
  For every graph $G$, there exists a weight $p$ on the vertices of the
  complementary graph $H = \overline{G}$, such that
  \[
    \alpha\bigl( G \boxtimes (\overline{G},p) \bigr) 
        = \vartheta\bigl( G \boxtimes (\overline{G},p) \bigr)
        = \vartheta(G)\,\vartheta(\overline{G},p).
  \]
\end{lem}

\medskip
Let us briefly recall the definition of weighted graphs 
and their graph invariants. A \emph{weighted graph} $(G,p)$ is a graph $G$ equipped with a
weight function $p:V\to \mathbb{R}_+$. The \emph{weighted independence number} 
$\alpha(G,p)$ is the largest total weight of an independent set in $G$, 
i.e. the largest sum of weights of the elements of an independent set. 
The \emph{weighted fractional packing number} of $(G,p)$ is likewise
\begin{equation}
  \label{eq:alpha-star-w}
  \alpha^*(G,p) = \max \sum_v p(v) t_v \text{ s.t. } t_v\geq 0\ \forall v,\ 
                          \sum_{v\in C} t_v \leq 1\ \forall \text{ cliques }C\subset V.
\end{equation}
Finally, the Lov\'asz number of a weighted graph is defined as 
\begin{align}
  \vartheta(G,p) = \max \tr B\Pi \text{ s.t. } B\geq 0,\ \tr B=1,\ B_{vw}=0\ \forall vw\in E,
\end{align}
where the matrix $\Pi$ has entries $\Pi_{vw} = \sqrt{p(v)p(w)}$;
cf.~the definition for unweighted graphs (\ref{eq:lovasz}). 

Note that for the constant-$1$ weight, $p(v) = 1$ for all $v$, which
we denote as $\mathbf{1}$,
the graph invariants attain the values of their unweighted versions:
\[
  \alpha(G) = \alpha(G,\mathbf{1}),\ 
  \alpha^*(G) = \alpha^*(G,\mathbf{1}),\ 
  \vartheta(G) = \vartheta(G,\mathbf{1}),\ etc.
\]
We will also consider (strong and disjunctive) products of weighted
graphs; their edge sets these are the same as those of the unweighted versions,
while the weights are multiplied pointwise: $(pp')(v,v')=p(v)p'(v')$.

\medskip\noindent
{\bf Dirac (bra-ket) notation.} In the rest of the paper we rely on the
following useful conventional notation for linear algebra,
called \emph{Dirac} or \emph{bra-ket notation} \cite{Dirac1939}: 
In a (real or complex) Hilbert space, the vectors are denoted 
$\ket{\psi}$, $\ket{\phi}$, etc. (``kets''), and the co-vectors -- which
are linear functions on the space -- are $\bra{\psi}$, $\bra{\phi}$, etc.
(``bras''), so that the inner product, denoted $\bra{\phi} \psi \rangle$
is at the same time the application of the co-vector
$\bra{\phi}$ to the vector $\ket{\psi}$, and can also be read as
the ordinary matrix product of the row vector $\bra{\phi}$
with the column vector $\ket{\psi}$. This extends to other
matrix products, such as $\bra{\phi} M \ket{\psi}$ 
for a linear operator/matrix $M$, and to outer products
$\ketbra{\psi}{\phi}$. In particular, the Hilbert space norm
is $\| \ket{\psi} \|_2 = \sqrt{\bra{\psi}\psi\rangle}$, and
for a unit vector $\ket{\psi}$,
$\proj{\psi}$ is the projector onto the line spanned by $\ket{\psi}$.
Note just one difference to usual mathematical convention: The
inner product $\bra{\phi} \psi \rangle$ is linear in the second
argument, and conjugate linear in the first. In practice this difference
will be unsubstantial for us, as the reader may assume real
Euclidean spaces throughout.

\medskip
\begin{proof}[Proof of Lemma~\ref{lem:alpha-theta}]
Let $G=(V,E)$ and let $\{\ket{\phi_v} : v\in V\}$ be an orthonormal representation 
of $\overline{G}$, i.e.~$\bra{\phi_v} \phi_w \rangle = 0$ for all $vw\in E$,
and $\ket{h}$ another unit vector (called the ``handle'' of the OR) such that 
$\vartheta(G) = \sum_{v \in V} |\bra{h} \phi_v \rangle |^2$; this is
another, equivalent characterization of the Lov\'{a}sz number~\cite{Lovasz1979}, 
cf.~\ref{app:optimization}. 
Equip the graph $\overline{G}$ with vertex weights $p(v)=| \bra{h} \phi_v \rangle |^2$.
Since the set $\{ (v,v) : v \in V \}$ is an independent set in $G \boxtimes \overline{G}$, 
it follows that
\begin{equation*}
\alpha(G \boxtimes (\overline{G},p)) \geq \sum_{v \in V} 1 \cdot p(v) 
                                  =    \sum_{v \in V} |\bra{h} \phi_v \rangle |^2 
                                  =    \vartheta(G).
\end{equation*}
Hence, 
\begin{equation}
  \label{chain}
  \vartheta(G) \leq \alpha(G \boxtimes (\overline{G},p))
               \leq \vartheta(G \boxtimes (\overline{G},p)) 
               =    \vartheta(G)\,\vartheta(\overline{G},p).
\end{equation}

On the other hand, the first characterization of the Lov{\'a}sz number 
$\vartheta$ of a weighted graph given in~\cite[Sec.~5]{Knu94} states that
\[
  \vartheta(\overline{G},p) = \min_{ \ket{\tilde{\phi}_v},\ket{\tilde{h}} } 
                                \left( \max_{v\in V} 
                                       \frac{p(v)}{|\bra{\tilde{h}} \tilde{\phi}_v\rangle|^2} \right),
\]
where the minimum is taken over all orthonormal representations and
handles of $\overline{G}$. 
Since $\{\ket{\phi_v} : v\in V \}$ with $\ket{h}$ is one candidate, 
a bound on the Lov{\'a}sz number of $(\overline{G},p)$ is
\begin{equation*}
  \vartheta(\overline{G},p) \leq \max_{v\in V} \frac{p(v)}{|\bra{h} \phi_v\rangle|^2} = 1.
\end{equation*}
Hence, $\vartheta(G)\,\vartheta(\overline{G},p) \leq \vartheta(G)$ and the 
inequalities in (\ref{chain}) turn into equalities, i.e.
\begin{equation}
  \label{equalities}
  \alpha(G \boxtimes (\overline{G},p)) 
                                     = \vartheta(G) \, \vartheta(\overline{G},p)
                                     = \vartheta(G \boxtimes (\overline{G},p)),
\end{equation}
as well as $\vartheta(\overline{G},p) = 1$, concluding the proof.
\end{proof}

\medskip
Now we come to our first main result of this paper; we show that
(\ref{eq:activation}) is attained \emph{asymptotically}.

\begin{theo}
  \label{thm:alpha-theta}
  For every graph $G$, 
  \[
    \sup_H \frac{\alpha(G \boxtimes H)}{\vartheta(G \boxtimes H)} = 1,
     \quad
     \text{or equivalently:}
     \quad
    \sup_H \frac{\alpha(G \boxtimes H)}{\vartheta(H)} = \vartheta(G).
  \]
\end{theo}

\medskip
Before proving this, we recall the definition of blow-up of
an integer-weighted graph, and a couple of auxiliary results from \cite{AFLS}:

\begin{defn}[{cf.~Ac\'{\i}n \emph{et al.}~\cite[Def.~A.2.9]{AFLS}}]
Let $(G,p)$ be a weighted graph with integer weights $p(v)\in\mathbb{N}_{>0}$ 
for all $v\in V$. 
Then the \emph{blow-up} $\mathrm{Blup}(G,p)$ is the unweighted graph with vertex set
\[
  V(p) := \big\{(v,i) \ :\ v\in V,\ i\in \{ 1,\ldots,p(v) \} \big\},
\]
where $(v,i)$ and $(w,j)$ are adjacent in $\mathrm{Blup}(G,p)$
if and only if $vw$ is an edge in $G$.
In other words, each vertex $v$ of $G$ is ``blown up'' to an independent 
set $\overline{K}_{p(v)}$.
\end{defn}

\begin{lem}[{Ac\'{\i}n \emph{et al.}~\cite[Lemma~A.2.7]{AFLS}}]
\label{incrweight}
Let $(G,p)$ be a weighted graph, $q\geq 0$ and 
$X\in\{\alpha,\Theta,\vartheta,\alpha^*\}$. Then,
\begin{equation}
  X(G,p) \leq X(G,r) \leq X(G,p + q\mathbf{1}) \leq X(G,p) + q|V|,
\end{equation}
for any weight $r$ with
$p(v) \leq r(v) \leq p(v)+q$ for all vertices $v$ of $G$.
\qed
\end{lem}

\begin{lem}[{Ac\'{\i}n \emph{et al.}~\cite[Lemma A.2.10]{AFLS}}]
For integer vertex weights $p(v) \in \mathbb{N}_{>0}$,
\label{blup}
\begin{enumerate}
\item $\mathrm{Blup}(G_1\boxtimes G_2,p_1\,p_2) = \mathrm{Blup}(G_1,p_1) \boxtimes \mathrm{Blup}(G_2,p_2)$;
\item $X(\mathrm{Blup}(G,p)) = X(G,p)$ for every $X\in\{\alpha,\Theta,\vartheta,\alpha^*\}$. \qed
\end{enumerate}
\end{lem}

\medskip
\begin{proof}[{Proof of Theorem~\ref{thm:alpha-theta}}]
For any two graphs $G$ and $H$, Lov\'{a}sz' fundamental inequality is
$\alpha(G\boxtimes H) \leq \vartheta(G\boxtimes H) = \vartheta(G)\,\vartheta(H)$,
so only the achievability of the opposite inequality by a sequence of
graphs $H$ has to be demonstrated.

We use Lemma~\ref{lem:alpha-theta}, giving us a weight $p:V \longrightarrow \RR_{\geq 0}$
such that $\alpha(G \boxtimes (\overline{G},p)) = \vartheta(G \boxtimes (\overline{G},p))$.
Now, consider the sequence of graphs 
$H_\ell := \textrm{Blup}(\overline{G},\lceil \ell\,p \rceil)$;
we claim that indeed, $\alpha(G \boxtimes H_\ell) \sim \vartheta(G \boxtimes H_\ell)$
as required.

To see this, multiply every term in (\ref{equalities}) by
an integer $\ell > 0$. Since the functions $\alpha$ and $\vartheta$ satisfy
$\ell\,X(G,p) = X(G,\ell\,p)$, it follows that
\begin{equation}
  \label{chainl}
  \alpha(G \boxtimes (\overline{G},\ell\,p)) 
                                          = \vartheta(G)\,\vartheta(\overline{G},\ell\,p).
\end{equation}
Now, by Lemma \ref{incrweight},
\begin{equation}
  \label{alpha1}
  \alpha(G \boxtimes (\overline{G},\ell\,p)) \leq \alpha(G \boxtimes (\overline{G},\lceil \ell\,p \rceil)) 
                                          \leq \alpha(G \boxtimes (\overline{G},\ell\,p+1)) 
                                          \leq \alpha(G \boxtimes (\overline{G},\ell\,p)) + |V|^2,
\end{equation}
and similarly
\begin{equation}
  \label{alpha2}
  \vartheta(\overline{G},\ell\,p) \leq \vartheta(\overline{G},\lceil \ell\,p \rceil)
                                  \leq \vartheta(\overline{G},\ell\,p+1)
                                  \leq \vartheta(\overline{G},\ell\,p) + |V|.
\end{equation}
In addition, Lemma \ref{blup} implies that 
\[
  \alpha(G \boxtimes H_\ell) = \alpha(G \boxtimes (\overline{G},\lceil \ell\,p \rceil)),
\]
hence putting this together with eqs.~(\ref{chainl}), (\ref{alpha1}) and (\ref{alpha2})
we get
\[
  \alpha(G \boxtimes H_\ell) \geq \vartheta(G)\,\vartheta(\overline{G},\ell\,p)
                          \geq \vartheta(G)\,\bigl( \vartheta(H_\ell) - |V| \bigr)
                          \geq \vartheta(G)\,\vartheta(H_\ell) - |V|^2.
\]
Since $\vartheta(H_\ell) \rightarrow \infty$ with growing $\ell$,
the claim follows.
\end{proof}

\medskip
\begin{remark}
From the proof, we see that
\[
  \sup_H \frac{a(G \boxtimes H)}{\vartheta(H)} = \vartheta(G),
\]
for any graph parameter $a(G\boxtimes H)$ in the numerator bounded between 
$\alpha(G\boxtimes H)$ and $\vartheta(G\boxtimes H)$, such as $\widetilde{\alpha}(G\boxtimes H)$,
the \emph{entanglement-assisted independence number} \cite{Beigi2010,CMRSSW}, 
Shannon's original zero-error capacity $\Theta(G\boxtimes H)$,
or Schrijver's variant $\vartheta^-(G\boxtimes H)$ of the Lov\'{a}sz number --- see 
Section~\ref{sec:theta-plus-minus} below.

This shows that the only upper bound on $\alpha$ that is (sub-)multiplicative
under strong graph products, and is at least as good as $\vartheta$, is the 
Lov\'{a}sz number itself.

We can also give an information theoretic interpretation of Theorem~\ref{thm:alpha-theta},
based on the recent discovery that $\vartheta(H)$ is precisely the zero-error
capacity assisted by no-signalling correlations, of quantum channels
with confusability graph $H$ \cite{DuanWinter2014}. Hence the quotient
$\frac{\alpha(G\boxtimes H)}{\vartheta(H)}$ is the ratio between how
much we can communicate through $G$ with the aid of some $H$ that we 
``borrow'', and the ``value'' of that other channel.
\end{remark}

\section{Duality of independence number and fractional packing number}
\label{sec:rosenfeld}
Taking inspiration from the second formulation of
Theorem~\ref{thm:alpha-theta}, we might wonder why we should
have the Lov\'{a}sz number in the denominator. Perhaps more
than one reader might object that it would be more natural to
compare like with like, i.e.~$\alpha$ with $\alpha$.

\begin{theo}[{Rosenfeld~\cite{Rosenfeld1967}, Hales~\cite{Hales1973}}]
  \label{thm:rosenfeld}
  For every pair of graphs $G=(V,E)$ and $H=(V',E')$,
  \begin{equation}
    \label{eq:rosenfeld}
    \alpha(G \boxtimes H) \leq \alpha^*(G)\,\alpha(H).
  \end{equation}
  Furthermore, this is tight for every $G$ and $H$ individually:
  Namely, there exist graphs $G'$ and $H'$ such that
  \begin{align}
    \label{eq:rosen-eq-1}
    \alpha(G \boxtimes H') &= \alpha^*(G)\,\alpha(H'), \\
    \label{eq:rosen-eq-2}
    \alpha(G' \boxtimes H) &= \alpha^*(G')\,\alpha(H).
  \end{align}
  In other words, for all graphs $G$,
  \[
    \max_H \frac{\alpha(G \boxtimes H)}{\alpha(H)} = \alpha^*(G),
    \quad
    \max_H \frac{\alpha(G \boxtimes H)}{\alpha^*(H)} = \alpha(G).
  \]
\end{theo}
\begin{proof}
All of this is (implicitly) included in the proof of \cite[Thm.~2]{Rosenfeld1967}.
We rephrase Rosenfeld's proof in our terms, which seems slightly 
more direct to us and is more geared towards our objective.

The first part is identical to Hales' proof of (\ref{eq:rosenfeld}) 
\cite[Thm.~4.2]{Hales1973}. Let $I \subset G\boxtimes H$ be an
independent set of maximum size $\alpha(G\boxtimes H)$. 
Define, for vertices $v\in V$,
\[
  f(v) := \frac{1}{\alpha(H)}\bigl| (\{v\}\boxtimes H) \cap I \bigr|.
\]
We claim that $f$ is a fractional packing of $G$. Indeed, for any
clique $C\subset G$, $I_C := (C\boxtimes H) \cap I$ is an independent set of
$C \boxtimes H$, which means that $I_C$ intersects each $C \boxtimes \{w\}$,
$w\in V'$, in at most one point. Hence,
\[
  J := \{ w : \exists v\in C\ (v,w) \in I \}
\]
is an independent set with $|I_C| = |J| \leq \alpha(H)$, and so
\[
  \sum_{v\in C} f(v) = \frac{1}{\alpha(H)}\bigl| (C\boxtimes H) \cap I \bigr| \leq 1.
\]
But now,
\[
  \alpha^*(G) \geq \sum_{v\in V} f(v) = \frac{1}{\alpha(H)}|I| = \frac{\alpha(G\boxtimes H)}{\alpha(H)},
\]
proving the inequality (\ref{eq:rosenfeld}).

Eq.~(\ref{eq:rosen-eq-2}) is trivial with $G'$ any complete graph.

To prove eq.~(\ref{eq:rosen-eq-1}), consider an optimal fractional
packing of $G$: $f(v) = \frac{n(v)}{N}$, with non-negative integers 
$N$ and $n(v)$; in particular, $\alpha^*(G) = \frac{1}{N}\sum_{v\in V} n(v)$.
[Recall that the fractional packing number is an LP, hence it has an
optimal solution consisting only of rational numbers.]
Now let $H' = \mathrm{Blup}(\overline{G},n)$, which we claim is the graph 
we are looking for. Indeed, 
\[
  \alpha(H') = \max_{{C\subset G \atop \text{clique}}} \sum_{v\in C} n(v) \leq N,
\]
the first identity by the observation that the maximal independent sets are
exactly the blow-ups of independent sets of $\overline{G}$, the second
inequality by the definition of a fractional packing.
On the other hand, because the blown-up diagonal 
$\{(v,(v,\ell)) :\ v\in V,\ 1\leq\ell\leq n(v) \}$
is an independent set in $G\boxtimes H'$, we have
\[
  \alpha(G\boxtimes H') \geq \sum_{v\in V} n(v) = N\,\alpha^*(G) \geq \alpha^*(G)\,\alpha(H').
\]
As we know the opposite inequality already, this concludes the proof.
\end{proof}

\medskip
It may be instructive, or entertaining, to view Theorems~\ref{thm:alpha-theta}
and \ref{thm:rosenfeld} as some kind of tight combinatorial H\"older inequalities:
The expression on the left hand side of 
eq.~(\ref{eq:rosenfeld}), which is a function of the graph product, is 
upper bounded by the product of functions of the factor graphs:
\[
  a(G \boxtimes H) \leq b(G)\,c(H).
\]
If for every graph $G$ ($H$) there exists an $H$ ($G$) making the above
an equality, or an asymptotic equality, we call $b$ and $c$ 
\emph{(asymptotically) dual with respect to $a$}, and the parameter
$a$ the \emph{pivot} of the duality.
Rosenfeld's Theorem~\ref{thm:rosenfeld} shows that $\alpha$ and $\alpha^*$
are dual with respect to $\alpha$, and Theorem~\ref{thm:alpha-theta}
says that $\vartheta$ is asymptotically self-dual with respect to $\alpha$.

We are thus led to consider more general upper bounds
on $\alpha(G \boxtimes H)$ in terms of products $b(G)\,c(H)$,
with special attention to dual pairs. We do not know as of
yet how to characterize all dual pairs for $\alpha$.
However, in the next section we shall show a third example.

\section{Duality of $\vartheta^-$ and $\vartheta^+$ with respect to $\alpha$}
\label{sec:theta-plus-minus}
Schrijver's variant $\vartheta^-$ \cite{Schrijver1979,MRR1978}
and Szegedy's variant $\vartheta^+$ \cite{Szegedy1994} of the Lov\'{a}sz number
are defined as follows:
\begin{align*}
  \vartheta^-(G) &= \max \tr BJ     \text{ s.t. } B\geq 0,\ \tr B=1,\ B_{vw}\geq 0\ \forall v,w,\ 
                                                                B_{vw}=0\ \forall vw\in E ,      \\
  \vartheta^+(G) &= \max \tr BJ     \text{ s.t. } B\geq 0,\ \tr B=1,\ B_{vw}\leq 0\ \forall vw\in E .
\end{align*}
(See~\ref{app:optimization} for equivalent characterizations 
and more properties of these two parameters.)
Then, we have

\begin{lem}
  \label{lem:theta-plus-minus}
  For any two graphs $G=(V,E)$ and $H=(V',E')$,
  \begin{equation}
    \label{eq:alpha-theta-pm-bound}
    \alpha(G\boxtimes H) 
                      \leq \vartheta^-(G \boxtimes H)
                      \leq \vartheta^-(G)\,\vartheta^+(H)
                      \leq \vartheta^+(G \ast H).
  \end{equation}
  In particular, for a graph $G$ on $n$ vertices and its complement $H=\overline{G}$,
  \[
    n \leq \alpha(G\boxtimes\overline{G}) \leq \vartheta^-(G)\,\vartheta^+(\overline{G}),
  \]
  with equality if $G$ is vertex-transitive.
\end{lem}
\begin{proof}
Schrijver and McEliece \emph{et al.} proved $\alpha \leq \vartheta^-$ 
\cite{Schrijver1979,MRR1978}.
The second and third inequality are proved via the primal and dual 
SDP characterizations of $\vartheta^{\pm}$.

\medskip
\underline{$\vartheta^-(G)\,\vartheta^+(H) \leq \vartheta^+(G \ast H)$}:
We use the primal SDPs given above, according to which we choose
feasible $B\geq 0$, $\tr B=1$ and $C\geq 0$, $\tr C=1$ for 
$\vartheta^-(G)$ and $\vartheta^+(H)$, respectively: 
$B_{vw} \geq 0$ for all $v,w$ and $B_{vw} = 0$ for $vw \in E$, and
$C_{v'w'} \leq 0$ for all $v'w' \in E'$. Then it is straightforward to
check that $B\ox C \geq 0$ is feasible for $\vartheta^+(G\ast H)$.

\medskip
\underline{$\vartheta^-(G \boxtimes H) \leq \vartheta^-(G)\,\vartheta^+(H)$}:
We use the dual SDP formulations of $\vartheta^{\pm}$, \ref{app:optimization},
eqs.~(\ref{eq:theta-minus-min}) and (\ref{eq:theta-plus-min}), according to
which we choose dual feasible $\lambda$ and $Y \geq J$ for $G$ and dual feasible
$\mu$ and $Z\geq J$ for $H$: $Y_{vv}=\lambda$ and $Z_{v'v'}=\mu$ for all $v$ and $v'$,
$Y_{vw} \leq 0$ for all $v\not\sim_G w$, $Z_{v'w'} \geq 0$ for all $v',w'$,
and $Z_{v'w'}=0$ for all $v'\not\sim_H w'$. It is straightforward to check that
the pair $\lambda\mu$ and $Y\ox Z \geq J\ox J$ is dual feasible for $\vartheta^-(G\boxtimes H)$.

\medskip
The case of $H=\overline{G}$ follows from $\alpha(G\boxtimes\overline{G}) \geq n$,
and is originally due to Szegedy \cite{Szegedy1994}, who also proved the equality 
in the vertex-transitive case.
\end{proof}

\medskip
\begin{remark}
Whereas $\vartheta$ is know to be multiplicative under both
the strong and the disjunctive product, this carries over
to $\vartheta^{\pm}$ only partially. Namely, it holds that
\begin{align*}
  \vartheta^-(G \boxtimes H) &\geq \vartheta^-(G)\,\vartheta^-(H), \\
  \vartheta^+(G \ast H)      &\leq \vartheta^+(G)\,\vartheta^+(H),
\end{align*}
but both inequalities can be strict, see~\cite[App.~A]{CMRSSW} for
explicit examples.

On the other hand, it is known that $\vartheta^-(G \ast H) = \vartheta^-(G)\,\vartheta^-(H)$
for all $G$ and $H$,
while the analogous $\vartheta^+(G \boxtimes H) = \vartheta^+(G)\,\vartheta^+(H)$
has been proven only for vertex-transitive $G$ and $H$, 
but is conjectured in general \cite[App.~A]{CMRSSW}.
\end{remark}

\medskip
The last part of Lemma~\ref{lem:theta-plus-minus} suggests the same 
question as for the Lov\'{a}sz number in Section~\ref{sec:activation}: 
Does there always exist a graph $H$, depending on $G$, such
that $\alpha(G\boxtimes H) = \vartheta^-(G)\,\vartheta^+(H)$? While we cannot
answer this question, we show that the answer is yes in an asymptotic
sense, building on a weighted analogue as before.

\begin{theo}
  \label{thm:alpha-theta-plus-minus}
  For every graph $G = (V,E)$, 
  \begin{align}
    \label{eq:alpha-theta-plus-minus}
    \sup_H \frac{\alpha(G \boxtimes H)}{\vartheta^+(H)} &=    \vartheta^-(G), \\
    \label{eq:alpha-theta-minus-plus}
    \sup_H \frac{\alpha(G \boxtimes H)}{\vartheta^-(H)} &=    \vartheta^+(G).
  \end{align}
\end{theo}
\begin{proof}
From Lemma~\ref{lem:theta-plus-minus} we know
\[
  \alpha(G\boxtimes H) \leq \vartheta^-(G)\,\vartheta^+(H),
\]
hence the inequality ``$\leq$'' in both eqs. (\ref{eq:alpha-theta-plus-minus})
and (\ref{eq:alpha-theta-minus-plus}) follows. In the vertex-transitive case
we have
\[
  \alpha(G\boxtimes\overline{G}) = |V| = \vartheta^-(G)\,\vartheta^+(\overline{G})
                                       = \vartheta^+(G)\,\vartheta^-(\overline{G}).
\]

The general proof of ``$\geq$'' in eq.~(\ref{eq:alpha-theta-plus-minus})
is similar to Theorem~\ref{thm:alpha-theta}: By Lemma~\ref{lem:alpha-theta-pm}
below, there exists a weight $p:V\longrightarrow \RR_{\geq 0}$ such that
\[
  \alpha(G \boxtimes (\overline{G},p)) = \vartheta^-(G)\,\vartheta^+(\overline{G},p).
\]
Now, letting $H_\ell = \textrm{Blup}(\overline{G},\lceil \ell\,p \rceil)$
does the trick, observing that Lemmas~\ref{incrweight} and \ref{blup}
extend to $\vartheta^{\pm}$. 

Analogously, to prove eq.~(\ref{eq:alpha-theta-minus-plus}), we use
Lemma~\ref{lem:alpha-theta-pm} below once more, showing that there exists a 
weight $q:V\longrightarrow \RR_{\geq 0}$ such that
\[
  \alpha(G \boxtimes (\overline{G},q)) = \vartheta^+(G)\,\vartheta^-(\overline{G},q).
\]
As before, letting $H_\ell = \textrm{Blup}(\overline{G},\lceil \ell\,q \rceil)$
does what we need, observing that Lemmas~\ref{incrweight} and \ref{blup}
extend to $\vartheta^{\pm}$.
\end{proof}

\begin{lem}
  \label{lem:alpha-theta-pm}
  For every graph $G$, there exists a weight $p$ on the vertices of the
  complementary graph $H = \overline{G}$, such that
  \[
    \alpha\bigl( G \boxtimes (\overline{G},p) \bigr) 
        = \vartheta^-(G)\,\vartheta^+(\overline{G},p).
  \]
  
  There also exists a weight $q$ on $H = \overline{G}$, such that
  \[
    \alpha\bigl( G \boxtimes (\overline{G},q) \bigr) 
        = \vartheta^+(G)\,\vartheta^-(\overline{G},q).
  \]
\end{lem}
\begin{proof}
As one might expect, this goes very similar to the proof of
Lemma~\ref{lem:alpha-theta}, using the characterizations of $\vartheta^{\pm}$
in \ref{app:optimization}.

For the first identity,
according to eq.~(\ref{eq:theta-minus-gram}),
we can find a non-negative orthonormal representation $\ket{\phi_v}$
of $\overline{G}$ (meaning that $\bra{\phi_v}\phi_w\rangle \geq 0$ for
all vertices $v,w$)
and a consistent unit vector $\ket{h}$
(meaning that $\bra{h}\phi_v\rangle \geq 0$ for all $v$), 
such that $\vartheta^-(G) = \sum_{v \in V} |\bra{h} \phi_v \rangle |^2$.
On the other hand, this non-negative OR is feasible for $\vartheta^+$
of the complementary graph, according to eq.~(\ref{eq:theta-plus-umbrella}),
and its weighted analogue. Hence, with $p(v) = |\bra{h} \phi_v \rangle |^2$,
we have $\vartheta^+(\overline{G},p) \leq 1$.
Now, as in the proof of Lemma~\ref{lem:alpha-theta}, the diagonal 
$\{(v,v) : v\in V \}$ is an independent set in $G\boxtimes (\overline{G},p)$,
with weight
\[\begin{split}
  \vartheta^-(G) =     \sum_{v \in V} |\bra{h} \phi_v \rangle |^2
                 &\leq \alpha\bigl(G\boxtimes (\overline{G},p)\bigr)       \\
                 &\leq \vartheta^-\bigl(G\boxtimes (\overline{G},p)\bigr)  \\
                 &\leq \vartheta^-(G)\,\vartheta^+(\overline{G},p)
                  \leq \vartheta^-(G),
\end{split}\]
where we have used the weighted version of Lemma~\ref{lem:theta-plus-minus}, 
and hence all of the above inequalities are identities.

For the second identity, we proceed very similarly. Indeed,
according to eq.~(\ref{eq:theta-plus-gram}), we can find an obtuse representation 
$\ket{\phi_v'}$ of $\overline{G}$ (meaning $\bra{\phi_v'}\phi_w'\rangle \leq 0$
for all edges $vw\in E$) and a consistent unit vector $\ket{h'}$, 
such that $\vartheta^+(G) = \sum_{v \in V} |\bra{h'} \phi_v' \rangle |^2$.
At the same time, this obtuse representation is feasible for
$\vartheta^-$ of the complementary graph, according to eq.~(\ref{eq:theta-minus-umbrella}),
and its weighted analogue. Hence, with $q(v) = |\bra{h'} \phi_v' \rangle |^2$,
we have $\vartheta^-(\overline{G},q) \leq 1$.
Now, as before, the diagonal $\{(v,v) : v\in V \}$ is an independent set 
in $G\boxtimes (\overline{G},q)$, with weight
\[\begin{split}
  \vartheta^+(G) =     \sum_{v \in V} |\bra{h'} \phi_v' \rangle |^2
                 &\leq \alpha\bigl(G\boxtimes (\overline{G},q)\bigr)       \\
                 &\leq \vartheta^-\bigl(G\boxtimes (\overline{G},q)\bigr)  \\
                 &\leq \vartheta^+(G)\,\vartheta^-(\overline{G},q)
                  \leq \vartheta^+(G),
\end{split}\]
where we have used the weighted version of Lemma~\ref{lem:theta-plus-minus}, 
and hence all of the above inequalities are identities.
\end{proof}

\section{Analogues for the chromatic number as pivot}
\label{sec:chromatic}
By the celebrated Sandwich Theorem, cf.~\cite{Knu94},
\[
  \alpha(G) \leq \vartheta(G) \leq \chi(\overline{G}) = \sigma(G),
\]
where $\chi$ is the chromatic number and $\sigma$ 
the \emph{clique covering number} of the graph $G$: 
$\chi(\overline{G}) = \sigma(G)$,
because each valid colouring of a graph is a partitioning,
or more generally covering, of its vertex sets by independent
sets, which are precisely the cliques in the complementary
graph. To avoid the awkward complements [observe 
$\overline{G\boxtimes H} = \overline{G}\ast\overline{H}$, so
we have $\chi(G \ast H) = \sigma(\overline{G} \boxtimes \overline{H})$
and $\chi(G \boxtimes H) = \sigma(\overline{G} \ast \overline{H})$], 
we will primarily present the following results in terms of 
the clique covering number, even though they may be better known 
or more attractive in their ``chromatic'' guise.

For all the other quantities introduced so far, there is a
veritable ``francesinha'':
\[
  \alpha(G) \leq \widetilde{\alpha}(G) \leq \vartheta^-(G)
            \leq \vartheta(G) \leq \vartheta^+(G) 
            \leq \alpha^*(G) \leq \sigma(G) = \chi(\overline{G}).
\]
For the clique covering/chromatic number, both strong and disjunctive 
product yield interesting asymptotics;
McEliece and Posner solved it for $\sigma(G^{\boxtimes n})$ \cite{McEliecePosner1971},
and Witsenhausen initiated the study of $\sigma(G^{\ast n})$ \cite{Witsenhausen1976}.

\medskip
We start with the strong graph product, for which the older 
literature offers a duality between clique covering number and
fractional packing/covering number:

\begin{theo}[Cf.~Hales~\cite{Hales1973}, McEliece/Posner~\cite{McEliecePosner1971}]
  \label{thm:chromatic-rosenfeld}
  For every pair of graphs $G=(V,E)$ and $H=(V',E')$,
  \begin{equation}
    \label{eq:hales}
    \sigma(G \boxtimes H) \geq \alpha^*(G)\,\sigma(H),
    \quad\text{i.e.}\quad
    \chi(G \ast H) \geq \alpha^*(\overline{G})\,\chi(H).
  \end{equation}
  Furthermore, this is (asymptotically) tight for every $G$ and $H$ individually.
  Namely, for all graphs $G$,
  \begin{align*}
    \inf_H \frac{\sigma(G \boxtimes H)}{\sigma(H)}   = \alpha^*(G),
    \quad&\text{i.e.}\quad
    \inf_H \frac{\chi(G \ast H)}{\chi(H)} = \alpha^*(\overline{G}), \\
    \min_H \frac{\sigma(G \boxtimes H)}{\alpha^*(H)} = \sigma(G),
    \quad&\text{i.e.}\quad
    \min_H \frac{\chi(G \ast H)}{\alpha^*(\overline{H})} = \chi(G).
  \end{align*}
\end{theo}
\begin{proof}
Hales' proof of eq.~(\ref{eq:hales}) is quite similar to the
proof of the Rosenfeld bound (\ref{eq:rosenfeld}), now using the dual LP
for $\alpha^*(G)$, eq.~(\ref{eq:alpha-sigma-star}): 
Consider a minimal clique covering $\mathcal{C}$ of $G\boxtimes H$, w.l.o.g.~only
using maximal cliques, which are of the form $C\boxtimes D$ for cliques 
$C\subset V$ and $D\subset V'$. Define
\[
  g(C) := \frac{1}{\sigma(H)}\bigl| \{ D : C\boxtimes D \in \mathcal{C} \} \bigr|,
\]
and confirm that it is a fractional covering of $G$. Indeed, for every vertex
$v$ of $G$, the set
\[
  \mathcal{D}(v) = \{ D : \exists v\in C\ \text{s.t.}\ C\boxtimes D\in\mathcal{C} \}
\]
is a clique covering of $H$, and so for all $v$,
\[
  \sum_{C\ni v} g(C) = \frac{1}{\sigma(H)}|\mathcal{D}(v)| \geq 1.
\]
On the other hand,
\[
  \alpha^*(G) \leq \sum_{C\text{ clique}} g(C) 
              =    \frac{|\mathcal{C}|}{\sigma(H)} 
              =    \frac{\sigma(G \boxtimes H)}{\sigma(H)}.
\]

Regarding the asymptotic tightness, the second claim is trivial,
taking any $H=\overline{K}_m$. For the first claim, recall
the result of \cite{McEliecePosner1971}, which is the first
step in the following:
\[\begin{split}
  \alpha^*(G) &= \inf_n \bigl( \sigma(G^{\boxtimes n}) \bigr)^{1/n} \\
        &=    \inf_n \left( \prod_{k=0}^{n-1} \frac{\sigma(G \boxtimes G^{\boxtimes k})}
                                                   {\sigma(G^{\boxtimes k})} \right)^{1/n} \\
        &\geq \inf_n \min_{0 \leq k \leq n-1} \frac{\sigma(G \boxtimes G^{\boxtimes k})}
                                                   {\sigma(G^{\boxtimes k})}               \\
        &=    \inf_n \frac{\sigma(G \boxtimes G^{\boxtimes n})}{\sigma(G^{\boxtimes n})}
         \geq \inf_H \frac{\sigma(G \boxtimes H)}{\sigma(H)}, 
\end{split}\]
and the latter we know already to be $\geq \alpha^*(G)$.
\end{proof}

\medskip
\begin{remark}
Comparing with Theorem~\ref{thm:rosenfeld} and its proof, only the 
Rosenfeld-Hales inequalities (\ref{eq:rosenfeld}) and (\ref{eq:hales})
are done in a similar fashion, but the achievability parts are very different.
Indeed, for $\alpha$ we carefully construct a graph $H$ by blowing
up the complement of $G$, attaining equality spot-on.
For $\sigma$ instead we simply consider the sequence $H=G^{\boxtimes k}$ and
get equality asymptotically. 

This raises two questions: First, whether for every $G$ there exists 
an $H$ with $\sigma(G \boxtimes H) = \alpha^*(G)\,\sigma(H)$? And
second, whether
\[
  \sup_n \frac{\alpha(G \boxtimes G^{\boxtimes n})}{\alpha(G^{\boxtimes n})} = \alpha^*(G)?
\]
Or to determine the limit, if it converges to some smaller value $\geq \Theta(G)$.
\end{remark}

\medskip
Going to the disjunctive product, which has more edges, hence 
smaller clique covering numbers, than the strong product,
we have the following relations involving Lov\'{a}sz $\vartheta$'s 
and variants:

\begin{theo}
  \label{thm:chromatic-thetas}
  For any graphs $G$ and $H$,
  \begin{align*}
    \sigma(G \ast H) \geq \vartheta(G)\,\vartheta(H),
    \quad&\text{i.e.}\quad
    \chi(\overline{G} \boxtimes \overline{H}) \geq \vartheta(G)\,\vartheta(H), \\
    \sigma(G \ast H) \geq \vartheta^+(G)\,\vartheta^-(H),
    \quad&\text{i.e.}\quad
    \chi(\overline{G} \boxtimes \overline{H}) \geq \vartheta^+(G)\,\vartheta^-(H).
  \end{align*}
  As a consequence, for every graph $G$,
  \begin{align*}
    \inf_H \frac{\sigma(G \ast H)}{\vartheta(H)}   &\geq \vartheta(G),  \\
    \inf_H \frac{\sigma(G \ast H)}{\vartheta^-(H)} &\geq \vartheta^+(G),
     \qquad
    \inf_H \frac{\sigma(G \ast H)}{\vartheta^+(H)}  \geq \vartheta^-(G),
  \end{align*}
  with the obvious equivalent expressions in terms of the chromatic number.
\end{theo}
\begin{proof}
According to the Sandwich Theorem, cf.~\cite{Knu94}, 
\[
  \sigma(G\ast H) \geq \vartheta(G\ast H) = \vartheta(G)\,\vartheta(H).
\]
In fact, it is even known that \cite{Szegedy1994}
\[
  \sigma(G\ast H) \geq \vartheta^+(G\ast H) \geq \vartheta^+(G)\,\vartheta^-(H),
\]
where we have invoked Lemma~\ref{lem:theta-plus-minus}.
\end{proof}

\medskip
\begin{remark}
We do not know whether any of the infima in Theorem~\ref{thm:chromatic-thetas}
is actually equal to the given lower bounds; but comparison with
Theorems~\ref{thm:alpha-theta} and \ref{thm:alpha-theta-plus-minus} suggests
this as a distinct possibility.

However, intrinsically perhaps most interesting is the question of
determining
\begin{equation}
  \label{eq:chi-activation}
  \zeta(G) := \inf_H \frac{\sigma(G \ast H)}{\sigma(H)} 
            = \inf_H \frac{\chi(\overline{G} \boxtimes H)}{\chi(H)},
\end{equation}
of which we can trivially say that it is not larger than the
\emph{Witsenhausen rate}~\cite{Witsenhausen1976}
\[
  R_W(\overline{G}) = \inf_n \bigl( \sigma(G^{\ast n}) \bigr)^{1/n}
                    = \inf_n \bigl( \chi(\overline{G}^{\boxtimes n}) \bigr)^{1/n},
\]
by considering $H = G^{\ast k}$. By analogy with 
Theorem~\ref{thm:rosenfeld}, one might expect some kind of 
fractional combinatorial parameter, but we are not even aware of
nontrivial lower bounds on (\ref{eq:chi-activation}).
\end{remark}

\section{Discussion}
\label{sec:discussion}
Many natural graph parameters arising as combinatorial
optimization problems, such as independence number or chromatic
number, are not generally multiplicative under graph products,
but due to their nature retain super-multiplicativity
($\alpha$ under the strong product) or sub-multiplicativity
($\sigma$, $\chi$ under both the strong and disjunctive product),
and this extends to numerical parameters such as 
$\vartheta^{\pm}$. Some few, concretely the fractional packing
and clique covering number, and the Lov\'{a}sz number
miraculously turn out to be multiplicative (the first under
strong products, the second under both strong and disjunctive
products).
For the others, there is the nontrivial problem of characterizing
the regularizations
\[
  \Theta(G) = \sup_n \bigl( \alpha(G^{\boxtimes n}) \bigr)^{1/n},
  \quad
  R_W(\overline{G}) = \inf_n \bigl( \sigma(G^{\ast n}) \bigr)^{1/n},
  \quad
  R^*(\overline{G}) = \inf_n \bigl( \sigma(G^{\boxtimes n}) \bigr)^{1/n},
\]
only the last of which is known: McEliece and Posner showed it to
equal the fractional packing number $\alpha^*(G)$ \cite{McEliecePosner1971}.

In the present paper, we diverted from this consideration
of the behaviour of graph parameters under the product of
many copies of $G$, and looked more broadly how they are affected by
products with a generic other graph $H$. After showing that
the Lov\'{a}sz number is asymptotically attained by the independence number
for every graph $G$ when activated by suitable graphs $H$, we embarked on 
a study of tight upper bounds on the independence number of
graph products in terms of products of individual, ``dual'', graph parameters.
We could give some examples of such pairs, but have not been able
to construct a general theory.

\medskip
There are many questions left to be answered. For example,
what are the pairs of dual graph parameters for $\widetilde\alpha$,
the entanglement-assisted independence number (beyond the self-dual
$\vartheta$)? 

Some of the most intriguing questions arise around dual pairs
of which one is the same function as the pivot; already the
determination of the other quantity in the bound, i.e.~for example
\[
  \sup_H \frac{\Theta(G\boxtimes H)}{\Theta(H)}
   \quad\text{or}\quad
  \sup_H \frac{\widetilde{\alpha}(G\boxtimes H)}{\widetilde{\alpha}(H)},
\]
is highly nontrivial. The first one is easily seen to be $\leq \alpha^*(G)$,
so the question is whether there is a gap; for the second one we do not
even have a target. In the same category falls the determination
of $\zeta(G)$ in eq.~(\ref{eq:chi-activation}). All these quantities
are of the type of \emph{potential capacities} -- cf.~\cite{YangWinter2014},
where they are studied in  detail for the ordinary classical, quantum,
private and other capacities of quantum channels.

\section*{Acknowledgments}
We thank
Jan Bouda, Tobias Fritz, Anthony Leverrier, Laura Man\v{c}inska,
Giannicola Scarpa, Simone Severini and Dan Stahlke
for various illuminating discussions.

AA is supported by the European Research Council (CoG ``QITBOX''), 
the AXA Chair in Quantum Information Science, the John Templeton Foundation, 
the Spanish MINECO (Severo Ochoa Grant SEV-2015-0522 and 
FOQUS FIS2013-46768-P), and the Generalitat de Catalunya (SGR875).
RD is supported in part by the Australian Research Council under 
Grant DP120103776 (together with AW) and by the National Natural Science Foundation 
of China under Grants 61179030; furthermore in part by an ARC 
Future Fellowship under Grant FT120100449.
DER is supported by the U.K.~EPSRC.
ABS was or is supported by the European Research Council (CoG ``QITBOX''
and AdG ``NLST''), Chist-Era project ``DIQIP'',
the Spanish MINECO (project FIS2010-14830), and the FPU:AP2009-1174 PhD grant.
AW is supported by the European Commission (STREP ``RAQUEL''), the European Research 
Council (AdG ``IRQUAT''), the Spanish MINECO (project FIS2008-01236) 
with the support of FEDER funds, and the Generalitat de Catalunya
(CIRIT project 2014-SGR-966).

\appendix

\section{Fractional and semidefinite relaxations of the independence number}
\label{app:optimization}
Here we collect several known, and a couple of new, useful characterizations of 
$\vartheta$, $\vartheta^-$ and $\vartheta^+$ as optimization
problems, in particular SDPs. The graph will always be an unweighted
graph $G=(V,E)$, although all of the formulas below have analogues
with weights, cf.~Knuth's \cite{Knu94}. Recall the Lov\'{a}sz convention
of denoting confusability of vertices as $v\sim w$, meaning equality ($v=w$)
or an edge ($vw \in E$). 
An \emph{orthonormal representation (OR)} of $G$ is an assignment of
unit vectors $\{ \ket{\phi_v} \}$ in some (real) vector space to all 
vertices $v\in V$, such that $\bra{\phi_v} \phi_w \rangle = 0$
for all $v \not\sim w$.

\bigskip
All of the following are from Lov\'{a}sz \cite{Lovasz1979}:
\begin{align}
  \label{eq:theta-norm}
  \vartheta(G) &= \max \| \1+T \| \text{ s.t. } \1+T \geq 0,\ T_{vw} = 0 \ \forall v\sim w    \\
  \label{eq:theta-max}
               &= \max \tr BJ     \text{ s.t. } B\geq 0,\ \tr B=1,\ B_{vw}=0\ \forall vw\in E \\
  \label{eq:theta-min}
               &= \min \lambda    \text{ s.t. } Z \geq J,\ Z_{vv}=\lambda\ \forall v,\ 
                                                           Z_{vw}=0\ \forall v\not\sim w      \\
  \label{eq:theta-umbrella}
               &= \min \left( \max_v \frac{1}{|\bra{h}\phi_v\rangle|^2} \right) \text{ s.t. } 
                       \{ \ket{\phi_v} \} \text{ is an OR of } G,\ \ket{h} \text{ unit vector}.
\end{align}
Here, $J$ is the all-ones matrix.

Observe that in eq.~(\ref{eq:theta-norm}), $\1+T$ is precisely the 
Gram matrix $\bigl[\bra{\phi_v}\phi_w\rangle\bigr]_{vw}$ of an orthonormal
representation of $\overline{G}$, and by the definition of the 
operator norm,
\begin{equation}
  \label{eq:theta-gram}
  \vartheta(G) = \| \1+T \| = \sum_v |\bra{h}\phi_v\rangle|^2,
\end{equation}
for an eigenvector $\ket{h}$ of the largest eigenvalue of
$\sum_v \proj{\phi_v}$, which has the same spectrum as $\1+T$.

\bigskip
There are analogous formulas for Schrijver's $\vartheta^-$, the second and third
are from \cite{Schrijver1979,MRR1978}, the fourth is due to de Carli Silva 
and Tun\c{c}el \cite[Cor.~4.2]{deCarliSilva-Tuncel2013}; see also \cite{deCarliSilva2013}:
\begin{align}
  \label{eq:theta-minus-norm}
  \vartheta^-(G) &= \max \| \1+T \| \text{ s.t. } \1+T \geq 0,\ T_{vw}\geq 0\ \forall v,w,\ 
                                                                T_{vw} = 0 \ \forall v\sim w    \\
  \label{eq:theta-minus-max}
                 &= \max \tr BJ     \text{ s.t. } B\geq 0,\ \tr B=1,\ B_{vw}\geq 0\ \forall v,w,\ 
                                                                B_{vw}=0\ \forall vw\in E       \\
  \label{eq:theta-minus-min}
                 &= \min \lambda    \text{ s.t. } Z \geq J,\ Z_{vv}=\lambda\ \forall v,\ 
                                                             Z_{vw} \leq 0\ \forall v\not\sim w \\
  \label{eq:theta-minus-umbrella}
                 &= \min \left( \max_v \frac{1}{|\bra{h}\phi_v\rangle|^2} \right) \text{ s.t. }
                         \{ \ket{\phi_v} \} \text{ is an obtuse rep.~of } G,\ 
                         \ket{h} \text{ consistent unit vector}.
\end{align}
Here, an \emph{obtuse representation} of $G$ is an assignment of
unit vectors $\{ \ket{\phi_v} \}$ to all vertices $v\in V$ such that
$\bra{\phi_v} \phi_w \rangle \leq 0$ for all $v \not\sim w$, and 
a unit vector $\ket{h}$ is called \emph{consistent} if $\bra{h}\phi_v\rangle \geq 0$
for all $v$ \cite{deCarliSilva-Tuncel2013}.
The first relation, eq.~(\ref{eq:theta-minus-norm}), is proved by equating
it with eq.~(\ref{eq:theta-minus-max}): Namely, observe that
\[
  \| \1+T \| = \max_{\ket{\beta}} \bra{\beta} (\1+T) \ket{\beta}
             = \max_{\ket{\beta}} \tr \proj{\beta}(\1+T)
             = \max_{\ket{\beta}} \tr \bigl( \proj{\beta} \circ (\1+T) \bigr)J,
\]
where the maximization is over unit vectors $\ket{\beta}$ with 
components $\beta_v$, $\proj{\beta}$ is the projection onto $\CC\ket{\beta}$,
and $\circ$ is the Schur/Hadamard (entry-wise) product of matrices.
To attain the maximum, w.l.o.g.~all vector components $\beta_v\geq 0$,
so all entries $\beta_v\beta_w$ of $\proj{\beta}$ are non-negative, and so 
$B = \proj{\beta} \circ (\1+T)$ is feasible for eq.~(\ref{eq:theta-minus-max}).
Conversely, any such $B$ we can write as $B = \proj{\beta} \circ (\1+T)$
with a unit vector $\ket{\beta}$ with non-negative components, and
a matrix $T$ feasible for eq.~(\ref{eq:theta-minus-norm}).

Note that in eq.~(\ref{eq:theta-minus-norm}), $\1+T$ is precisely the 
Gram matrix $\bigl[\bra{\phi_v}\phi_w\rangle\bigr]_{vw}$ of a 
\emph{non-negative} orthonormal representation of $\overline{G}$, 
i.e.~$\bra{\phi_v} \phi_w \rangle \geq 0$ for all $v$ and $w$ and
$\bra{\phi_v} \phi_w \rangle = 0$ for $vw\in E$.
By the definition of the operator norm,
\begin{equation}
  \label{eq:theta-minus-gram}
  \vartheta^-(G) = \| \1+T \| = \sum_v |\bra{h}\phi_v\rangle|^2
\end{equation}
for an eigenvector $\ket{h}$ of the largest eigenvalue of
$\sum_v \proj{\phi_v}$, which has the same spectrum as $\1+T$. 
Furthermore, one may assume $\bra{h}\phi_v\rangle \geq 0$ for all $v\in V$.
This is due to the Perron-Frobenius theorem~\cite{HornJohnson}, 
which guarantees that the Gram matrix $\bigl[\bra{\phi_v}\phi_w\rangle\bigr]_{vw}$
has a unit eigenvector $\ket{\mu} = \sum_v \mu_v\ket{v}$ with non-negative
entries $\mu_v$ for the largest eigenvalue $\theta = \| \1+T \|$:
\[
  \| \1+T \| = \sum_{uv} \mu_v \mu_w \bra{\phi_v} \phi_w \rangle = \bra{X}X\rangle,
\]
with $\ket{X} = \sum_v \mu_v \ket{\phi_v} =: \sqrt{\theta}\ket{h}$. 
By construction, $\bra{h}\phi_v\rangle \geq 0$, and one can check by
direct calculation that 
\[
  \left( \sum_v \proj{\phi_v} \right) \ket{h} = \theta \ket{h}.
\]

\bigskip
For Szegedy's $\vartheta^+$ \cite{Szegedy1994}, instead, we have:
\begin{align}
  \label{eq:theta-plus-norm}
  \vartheta^+(G) &= \max \| \1+T \|_+ \text{ s.t. } \1+T \geq 0,\ T_{vw} \leq 0 \ \forall v\sim w    \\
  \label{eq:theta-plus-max}
                 &= \max \tr BJ     \text{ s.t. } B\geq 0,\ \tr B=1,\ B_{vw}\leq 0\ \forall vw\in E \\
  \label{eq:theta-plus-min}
                 &= \min \lambda    \text{ s.t. } Z \geq J,\ Z_{vv}=\lambda\ \forall v,\ 
                                                             Z_{vw}\geq 0\ \forall v,w,\ 
                                                             Z_{vw}= 0\ \forall v\not\sim w,        \\
  \label{eq:theta-plus-umbrella}
                 &= \min \left( \max_v \frac{1}{|\bra{h}\phi_v\rangle|^2} \right) \text{ s.t. }
                         \{ \ket{\phi_v} \} \text{ is a non-negative OR of } G,\ 
                         \ket{h} \text{ consistent unit vector}.
\end{align}
Here, $\| X \|_+ = \max_{\ket{\psi}\text{ positive}} |\bra{\psi} X \ket{\psi}|$
is the maximum overlap of the matrix $X$ with a \emph{positive} unit
vector $\ket{\psi} = \sum_v \psi_v \ket{v}$, i.e.~one with non-negative 
entries $\psi_v$. Note that $\| \cdot \|_+$ is a norm on matrices,
which we call the \emph{positive operator norm}.
Eqs.~(\ref{eq:theta-plus-max}) and (\ref{eq:theta-plus-min}) are the
original definitions from~\cite{Szegedy1994}, while eq.~(\ref{eq:theta-plus-umbrella})
is stated in~\cite[Prop.~2.1]{Szegedy1994}. [Note that the latter was
claimed without the requirement that $\bra{h}\phi_v\rangle\geq 0$ for all 
$v\in V$, but this comes out naturally from the equivalence proof with 
(\ref{eq:theta-plus-min}).]

To prove eq.~(\ref{eq:theta-plus-norm}), we start from 
eq.~(\ref{eq:theta-plus-max}) and observe
\[
  \| \1+T \|_+ = \max_{\ket{\beta}} \bra{\beta} (\1+T) \ket{\beta}
               = \max_{\ket{\beta}} \tr \proj{\beta}(\1+T)
               = \max_{\ket{\beta}} \tr \bigl( \proj{\beta} \circ (\1+T) \bigr)J,
\]
where the maximization is over unit vectors $\ket{\beta}$ with 
components $\beta_v \geq 0$. Thus, all entries 
$\beta_v\beta_w$ of the matrix $\proj{\beta}$ are non-negative, and so 
$B = \proj{\beta} \circ (\1+T)$ is feasible for eq.~(\ref{eq:theta-plus-max}).
Conversely, any such $B$ we can write as $B = \proj{\beta} \circ (\1+T)$
with a unit vector $\ket{\beta}$ with non-negative components, and
a matrix $T$ feasible for eq.~(\ref{eq:theta-plus-norm}).

Note that in eq.~(\ref{eq:theta-plus-norm}), $\1+T$ is precisely the 
Gram matrix $\bigl[\bra{\phi_v}\phi_w\rangle\bigr]_{vw}$ of an
obtuse representation of $\overline{G}$, 
i.e.~unit vectors with $\bra{\phi_v} \phi_w \rangle \leq 0$ for $vw\in E$.

In Section~\ref{sec:theta-plus-minus}, we need the following formulation
of $\vartheta^+$:
\begin{equation}
  \label{eq:theta-plus-gram}
  \vartheta^+(G) = \max \sum_v |\bra{h}\phi_v\rangle|^2 \text{ s.t. }
                         \{ \ket{\phi_v} \} \text{ is an obtuse rep.~of } \overline{G},\ 
                         \ket{h} \text{ consistent unit vector}.
\end{equation}

\begin{proof}
Let $B$ be an optimal solution in eq.~(\ref{eq:theta-plus-max}) for $G$. 
Since $B$ is positive semidefinite, there exist vectors $\ket{\psi_v}$ for 
$v \in V$, such that $B_{vw} = \bra{\psi_v}\psi_w\rangle$.  
Let $\ket{\Psi} = \sum_v \ket{\psi_v}$ and note that 
$\bra{\Psi}\Psi\rangle = \sum_{vw} B_{vw} = \tr BJ$ is the objective function
value of the solution $B$. 
Furthermore, $\bra{\psi_v}\Psi\rangle$ is the $v$-th row sum of $B$. 
Let $\ket{\phi_v} = \frac{1}{\| \ket{\psi_v} \|_2}\ket{\psi_v}$.
(If the numerator is $0$, let $\ket{\phi_v}$ be a unit vector orthogonal to 
all others, by moving to a higher dimension if necessary.)
Also, let $\ket{h} = \frac{1}{\| \ket{\Psi} \|_2}\ket{\Psi}$. 
We will show that this is a solution for~(\ref{eq:theta-plus-gram}) 
of value at least $\tr BJ$.

First, if $vw \in E$ is an edge, then
\[
  \bra{\phi_v}\phi_w\rangle = \frac{\bra{\psi_v}\psi_w\rangle}{\|\ket{\psi_v}\|_2 \|\ket{\psi_w}\|_2} 
                            = \frac{B_{vw}}{\|\ket{\psi_v}\|_2 \|\ket{\psi_w}\|_2} \leq 0,
\]
so we have an obtuse representation. Second,
\[
  \bra{h}\phi_v\rangle = \frac{\bra{\Psi}\psi_v\rangle}{\|\ket{\Psi}\|_2 \|\ket{\psi_v}\|_2} \geq 0,
\]
since $\bra{\Psi}\psi_v\rangle$ is the $v$-th row sum of $B$ which is nonnegative by 
Lemma~\ref{lem:nonneg} below. Thus, $\ket{h}$ is a consistent vector for the
obtuse representation.
Note that we should be a bit careful and point out that for 
$\ket{\psi_v} = 0$ the above inner products are $0$ by our choice of $\ket{\phi_v}$.
Third, the objective function:
Let $S$ be the set of indices of the nonzero rows of $B$, 
i.e.~$S=\{v: \ket{\psi_v} \neq 0\}$. We now have, using $B_{vv} \geq 0$,
\[
  \sum_v |\bra{h}\phi_v\rangle|^2 
           = \sum_{v \in S} \frac{|\bra{\Psi}\psi_v\rangle|^2}{\|\ket{\Psi}\|_2^2 \|\ket{\psi_v}\|_2^2}
           = \frac{1}{\tr BJ} \sum_{v \in S} B_{vv} \left|\frac{\bra{\Psi}\psi_v\rangle}{B_{vv}}\right|^2.
\]
Noting furthermore $\sum_v B_{vv} = 1$, we can use
Jensen's inequality to the convex function $x^2$, to obtain
\begin{align*}
\sum_v |\bra{h}\phi_v\rangle|^2 
           &\geq \frac{1}{\tr BJ} \left( \sum_{v \in S} B_{vv} \frac{\bra{\Psi}\psi_v\rangle}{B_{vv}} \right)^2 \\
           &=    \frac{1}{\tr BJ} \left( \sum_{v \in S} \bra{\Psi}\psi_v\rangle \right)^2 
            =    \frac{1}{\tr BJ} \bra{\Psi}\Psi\rangle^2 
            =    \tr BJ.
\end{align*}
This proves that the optimal solution to~(\ref{eq:theta-plus-gram}) 
is at least as large as the optimal solution to~(\ref{eq:theta-plus-max}). 

We now prove the opposite inequality.
Let $\ket{h}$, $\ket{\phi_v}$ for $v \in V$ be a solution 
for~(\ref{eq:theta-plus-gram}) of value $\theta := \sum_v |\bra{h}\phi_v\rangle|^2$. 
Define
\[
  \ket{\psi_v} = \frac{1}{\sqrt{\theta}} \proj{\phi_v}\,\ket{h},
\]
and let $B$ be the Gram matrix of these $\ket{\psi_v}$,
i.e.~$B_{vw}=\bra{\psi_v}\psi_w\rangle$.
We will show that $B$ is a solution for~(\ref{eq:theta-plus-max}) 
of value at least $\theta$. To start, as a Gram matrix, it is positive
semidefinite.

First, for an edge $vw\in E$, we have that
\[
  B_{vw} =    \bra{\psi_v}\psi_w\rangle 
         =    \frac{1}{\theta} \bra{h}\phi_v\rangle \bra{\phi_v}\phi_w \rangle \bra{\phi_w} h\rangle
         \leq 0,
\]
since $\bra{h}\phi_v\rangle \geq 0$ for all $v$ 
and $\bra{\phi_v}\phi_w \rangle \leq 0$ for $vw\in E$. 
Second,
\[
  \tr B = \sum_v \bra{\psi_v}\psi_v\rangle 
        = \frac{1}{\theta} \sum_v |\bra{h}\phi_v\rangle|^2 
        = 1.
\]
Finally, letting $M = \sum_v \proj{\phi_v}$, we have
\[\begin{split}
  \tr BJ  =    \sum_{vw} \bra{\psi_v}\psi_w\rangle                                        
         &=    \frac{1}{\theta} \sum_{vw} \langle h \proj{\phi_v} \proj{\phi_w} h \rangle \\
         &=    \frac{1}{\theta} \bra{h} M^2 \ket{h}                                       \\
         &\geq \frac{1}{\theta} \bra{h} M \proj{h} M \ket{h}                              
          =    \frac{1}{\theta} \left( \sum_v |\bra{h}\phi_v\rangle|^2 \right)^2
          =    \theta,
\end{split}\]
where in the third line we have used $\proj{h} \leq \1$,
hence $M^2= M \1 M \geq M\proj{h}M$.
This proves that the optimal solution to~(\ref{eq:theta-plus-max}) 
is at least as large as the optimal solution to~(\ref{eq:theta-plus-gram}),
concluding the proof. 
\end{proof}

\begin{lem}
  \label{lem:nonneg}
  If $B$ is an optimal solution to~(\ref{eq:theta-plus-max}), 
  then the row sum of any nonzero row of $B$ is positive.
\end{lem}
\begin{proof}
Suppose, by contradiction, that $B$ is an optimal solution 
to~(\ref{eq:theta-plus-max}) and that the $v$-th row of $B$ is 
nonzero and has non-positive row sum. Since $B \geq 0$, we have that $B_{vv} \geq 0$,
with equality if and only if the $v$-th row is all zero. 
This implies that $B_{vv} > 0$ and thus $\sum_{w \ne v} B_{vw} \leq - B_{vv} < 0$. 
Therefore, changing both the $v$-th row and column to zeros strictly 
increases the sum of the entries of $B$, while decreasing the trace. 
Note that this change keeps $B$ positive semidefinite, since it is 
equivalent to changing a vector in the Gram representation of $B$ to 
the zero vector. Therefore we can positively scale this new matrix to 
have trace $1$ and greater sum of all entries, giving us a better solution 
to~(\ref{eq:theta-plus-max}), a contradiction.
\end{proof}

%
%

\medskip
In this paper, we also looked at weighted Lov\'{a}sz numbers $\vartheta(G,p)$
and variants $\vartheta^{\pm}(G,p)$. These are defined by replacing the 
all-ones matrix $J$ in eqs.~(\ref{eq:theta-max}), (\ref{eq:theta-minus-max}) 
and (\ref{eq:theta-plus-max}) by the weights matrix 
$\Pi = \bigl[ \sqrt{p(v)p(w)} \bigr]_{vw}$. The other formulas are
changed accordingly; in particular eqs.~(\ref{eq:theta-umbrella}), 
(\ref{eq:theta-minus-umbrella}) and (\ref{eq:theta-plus-umbrella})
simply receive the weight $p(v)$ in the numerator.

\bigskip
Finally, we record here the mutually dual LPs of fractional packing and 
fractional clique covering:
\begin{equation}\begin{split}
  \label{eq:alpha-sigma-star}
  \alpha^*(G) &= \max \sum_v t_v \text{ s.t. } t_v\geq 0\ \forall v,\ 
                          \sum_{v\in C} t_v \leq 1\ \forall \text{ cliques }C\subset V \\
              &= \min \sum_C s_C \text{ s.t. } s_C\geq 0\ \forall \text{ cliques }C\subset V,\ 
                          \sum_{C\ni v} s_C \geq 1\ \forall v.
\end{split}\end{equation}
In particular, $\alpha(G) \leq \alpha^*(G) \leq \sigma(G)$.

Cf.~the very nice book \cite{ScheinermanUllman1997} for details on these,
where it is also discussed that a natural notion of 
\emph{fractional colouring} and \emph{fractional chromatic number}
leads to the same LP.

One thing we can check easily is the multiplicativity of $\alpha^*$
under strong graph products: 
$\alpha^*(G \boxtimes H) = \alpha^*(G)\,\alpha^*(H)$.
Indeed, since the product of primal feasible solutions
for $\alpha^*(G)$ and $\alpha^*(H)$ is feasible for $\alpha^*(G\boxtimes H)$,
we obtain ``$\geq$''. Likewise, ``$\leq$'' follows by observing that
the product of dual feasible solutions of the two graphs is dual
feasible for $\alpha^*(G\boxtimes H)$.


\bibliographystyle{elsarticle-num}


\end{document}